\documentclass[11pt]{article}
\usepackage{amsmath,amsthm,amsfonts,amssymb,mathrsfs,bm}
\usepackage{amsmath,amsthm,amsfonts,amssymb,bm,wasysym}
\usepackage{epsfig}
\usepackage[usenames]{color}
\usepackage{verbatim}
\usepackage{hyperref}
\usepackage{multicol}
\usepackage{comment}
\usepackage{float}
\usepackage{color}
\usepackage{enumerate}
\usepackage[normalem]{ulem}

\usepackage[color=green, textsize=tiny]{todonotes}

\parskip \medskipamount
\numberwithin{equation}{section}

\newcommand{\bca}{ {\rm BCap}}

\newtheorem{dfn}{Definition}[section]
\newtheorem{thm}[dfn]{Theorem}
\newtheorem{lem}[dfn]{Lemma}
\newtheorem{cor}[dfn]{Corollary}
\newtheorem{rem}[dfn]{Remark}

\newtheorem{prop}[dfn]{Proposition}\makeatletter

\title{Branching capacity of a random walk range}

\author{Bruno Schapira\thanks{Aix-Marseille Universit\'e, CNRS, I2M, UMR 7373, 13453 Marseille, France;  bruno.schapira@univ-amu.fr}}

\begin{document}

\date{}
\maketitle

\begin{abstract}
We consider the branching capacity of the range of a simple random walk on $\mathbb Z^d$, with $d\ge 5$, and show that it falls in the same universality class as the volume and the capacity of the range of simple random walks and branching random walks. To be more precise we prove a law of large numbers in dimension $d\ge 6$, with a logarithmic correction in dimension $6$, and identify the correct order of growth in dimension $5$. The main original part is the law of large numbers in dimension $6$, 
for which one needs a precise asymptotic of the non-intersection probability of an infinite invariant critical tree-indexed walk with a two-sided simple random walk. The result is analogous to the estimate proved by Lawler for the non-intersection probability of an infinite random walk with a two-sided walk in dimension four. While the general strategy of Lawler's proof still applies in this new setting, many steps require new ingredients. 
\newline
MSC 2020 \emph{subject classifications.} 60F15; 60J80.
\end{abstract}
 
\tableofcontents

\section{Introduction} 
We start by recalling some important definitions and we will then state our main results. 
The branching capacity is defined here in terms of an offspring distribution $\mu$ on $\mathbb N$, which is fixed in the whole paper and assumed to be critical, in the sense that 
$\sum_ i i\mu(i)=1$. We further assume that it has a finite and positive variance $\sigma^2$. 
We write the size biased distribution of $\mu$ as $\mu_{{\rm sb}}$, which we recall is defined by $\mu_{{\rm sb}}(i)=i\mu(i)$, for all $i\ge 0$. 

We then consider $\mathcal T$ an infinite spatial tree constructed as follows: 
\begin{itemize}
	\item The root produces $i$ offspring with probability $\mu(i-1)$ for every $i\ge 1$. The first offspring of the root is \emph{special}, while the others if they exist are \emph{normal}. 
	\item Special vertices produce offspring independently according to $\mu_{{\rm sb}}$, while normal vertices produce offspring independently according to $\mu$. 
	\item One of the offspring of a special vertex is chosen at random to be a special vertex, while the rest are normal ones. 
\end{itemize}
By construction $\mathcal T$ has a unique infinite path emanating from the root that we call \emph{spine}. We assign label $0$ to the root. We assign positive labels to the vertices to the right of the spine according to depth first search from the root and we assign negative labels to the vertices to the left of the spine and the spine vertices as well according to depth first search from infinity. We call the vertices with negative labels (including the spine vertices) the \emph{past} of $\mathcal T$ and denote them $\mathcal T_-$, while the vertices with non-negative labels are in the 
\emph{future} of $\mathcal T$ and we denote them $\mathcal T_+$. Note that the root does not have any offspring in the past of $\mathcal T$.

Given $x\in \mathbb Z^d$, we denote by $(S_u^x)_{u\in \mathcal T}$ the random walk indexed by $\mathcal T$, starting from $x$, whose jump distribution is the uniform measure on the neighbors of the origin, and denote its range in the past by 
$$\mathcal T^x_- = \{S_u^x : u\in \mathcal T_-\}.$$  
The equilibrium measure $e_A$ of a finite set $A\subset \mathbb Z^d$, with $d\ge 5$, has been introduced by Zhu~\cite{Zhu16}, and is defined by,  
$$e_A(x) = \mathbf 1\{x\in A\} \cdot \mathbb P(\mathcal T^x_- \cap A = \emptyset ). $$ 
Then the branching capacity of a finite set $A$ is defined similarly as the usual Newtonian capacity, namely  
$$\textrm{BCap}(A) = \sum_{x\in A} e_A(x). $$ 

Consider now $(X_n)_{n\ge 0}$ an independent simple random walk on $\mathbb Z^d$ (i.e. a random walk whose law of increments is the uniform measure on the neighbors of the origin), and define its range at time $n$ as 
$$\mathcal R_n=\{X_0,\dots, X_n\}.$$ 
Our main object of study in this paper is the branching capacity of the range $\bca(\mathcal R_n)$, in dimension $d\ge 5$, and our goal is to show that it satisfies the same universal asymptotic behavior as the volume~\cite{DE51} and the capacity~\cite{ASS18,ASS19,Chang17,JO69} of the range, with only a shift of the critical dimension of respectively two and four units, which is here the dimension $6$. Interestingly, the same universal results have also been proved recently for 
the volume~\cite{LGL15,LGL16} and the capacity~\cite{BW20,BH22} of a critical branching random walk, and of course it would be of interest to see if they can as well be extended to the branching capacity of a branching random walk, but we leave this for a future work.  

\vspace{0.2cm}
Our first result is a strong law of large numbers. The proof is entirely similar to the one for the usual Newtonian capacity, which dates back to Jain and Orey~\cite{JO69}, and is reproduced at the end of this paper for reader's convenience (to be more precise the fact that the limiting constant is positive requires a specific argument). 

\begin{thm}\label{m7}
Assume $d \ge 7$. There exists a constant $c_d>0$, such that almost surely, 
\begin{equation}\label{LLN}
\lim_{n\to \infty} \frac{\bca(\mathcal R_n)}{n}  = c_d.  
\end{equation}
\end{thm}
It is very likely that a central limit theorem, with the usual renormalization in $\sqrt n$, could be proved in dimension $d\ge 8$, following the same lines as in~\cite{ASS18}. 
In dimension $7$ it is expected that a logarithmic correction should appear in the normalization, 
but this might be a much more challenging problem, as the corresponding results in the simpler cases of the volume and the capacity of the range of a random walk are already quite involved, see~\cite{JP71,S20} respectively.

The main contribution of this paper is the law of large numbers in dimension $6$, which requires some more original work. We only present here a detailed proof of the weak law (with a convergence in probability), but a strong law (with an almost sure convergence) could be proved as well without much additional work, see Remark~\ref{rem.stronglaw} for more details. The main step is to obtain the asymptotic of 
the expected branching capacity of the range. The general strategy for this is the same as for the capacity of the range, in which case the corresponding result follows from the estimates proved by Lawler~\cite{L91} for the non-intersection probability between one walk and another independent two-sided walk in dimension four, see~\cite{ASS18,Chang17}. However, 
one serious issue that arises when working with the tree-indexed walk is the lack of Markov property, which in particular 
has for damaging consequence that there is no simple last exit formula as one has for a simple random walk.  
This leads to some non-trivial complications, which fortunately can be overtaken.

\begin{thm}\label{m6}
Assume $d= 6$, and that $\mu$ has a finite third moment. Then one has the convergence in probability and in $L^2$,   
\begin{align*}
\lim_{n\to \infty} \frac{\log n}{n} \cdot \bca (\mathcal R_n) = \frac{2\pi^3}{27\sigma^2}.  
\end{align*}
\end{thm}
Of course a natural question now would be to prove a central limit theorem, as it was done in~\cite{LG86,ASS19} respectively for the volume and the 
capacity of the range. We leave this for future work, as it would require some really new ingredients, in particular one major issue would be to identify a simple expression for the term 
$$\chi(A,B) = \bca(A\cup B) - \bca(A) - \bca(B),$$
where $A$ and $B$ are arbitrary finite subsets of $\mathbb Z^d$, and improve the bounds that we have on the variance of $\bca (\mathcal R_n)$. 

\vspace{0.2cm}
To conclude we provide bounds identifying the correct order of growth of the expected branching capacity of the range in dimension five. The upper bound is easily obtained 
by using monotonicity of the branching capacity, and known bounds on the branching capacity of balls. The lower bound is more difficult, and we rely here on a recent result  of~\cite{ASS23} showing a variational characterization of the branching capacity. 
\begin{prop}\label{m5}
For $d=5$, there exist positive constants $c_5^-$ and $c_5^+$, such that for all $n\ge 1$, 
\begin{align*}
 c_5^-\cdot \sqrt n \le \mathbb E\left[\bca (\mathcal R_n)\right]\le c_5^+\cdot \sqrt n. 
\end{align*}
\end{prop}

The paper is organized as follows. In Section~\ref{sec.preliminaries} we prove some preliminary results which could be of general interest. 
In particular we prove an analogous version in the setting of branching random walks of a key equation discovered by Lawler, relating some non-intersection events 
and a sum of Green's function along the positions of a random walk, see Lemma~\ref{lem.key.lawler} and Corollary~\ref{cor.Lawler}. We also prove there 
some quantitative bounds on the speed of convergence toward the branching capacity of a set $A$, of the (conveniently normalized) probability to hit $A$ for a tree-indexed walk, as the starting point goes to infinity, see Lemma~\ref{lem.hit}. Then Section~\ref{sec.6} focuses on the case of dimension $6$, and we prove there Theorem~\ref{m6}, 
while short proofs of Theorem~\ref{m7} and Proposition~\ref{m5} are given in Section~\ref{sec.57}.

\section{Preliminaries}\label{sec.preliminaries}
\subsection{Some additional notation}
We let $\|x\|$ denote the Euclidean norm of $x\in \mathbb Z^d$. For $m\ge 0$, we denote by $B(0,m)$ the closed Euclidean ball of radius $m$ centered at the origin (intersected with $\mathbb Z^d$), and for a set $\Lambda\subset \mathbb Z^d$, we let $\partial \Lambda$ be the inner boundary of $\Lambda$ consisting of points of $\Lambda$ having at least one neighbor outside $\Lambda$. We denote by $|A|$ the size of a finite set $A\subset \mathbb Z^d$.

\vspace{0.1cm}
We use here the convention that a Geometric random variable $X$ with parameter $p\in (0,1)$ takes values in $\mathbb N$, and is such that for any $k\ge 0$, $\mathbb P(X=k) = p(1-p)^k$. 

\vspace{0.1cm}
We denote by $(\widetilde X_k)_{k\ge 0}$ the simple random walk indexed by the vertices of the spine, equipped with its intrinsic labelling (i.e. the vertex on the spine at graph distance $k$ from the root has intrinsic label $k$). The law of a simple random walk starting from $x$ is denoted by $\mathbb P_x$ while the corresponding expectation is denoted by $\mathbb E_x$, and we abbreviate them in $\mathbb P$ and $\mathbb E$ respectively when the walk starts from the origin. For $n\le m$, we write the range of a random walk $(X_k)_{k\in \mathbb N}$ between times $n$ and $m$ as  
$\mathcal R[n,m] = \{X_n,\dots,X_m\}$. 

\vspace{0.1cm}
The root of the tree $\mathcal T$ is denoted by $\emptyset$. 

\vspace{0.1cm}
Given two functions $f$ and $g$, we write $f\lesssim g$, or sometimes also $f=\mathcal O(g)$, if there exists a constant $C>0$, such that $f(x) \le Cg(x)$ for all $x$, and likewise write $f\gtrsim g$, if $g\lesssim f$. 
We write $f= o(g)$ if $f(x)/g(x)$ goes to $0$ as $x$ goes to infinity, and $f\sim g$, when $|f-g|=o(g)$.  

\subsection{An exact last passage formula}
Our tree-indexed random walks are not Markovian, but nevertheless they satisfy a certain last passage formula, which takes the following form. 
\begin{lem}[Last passage formula] \label{lem.lastpassage}
For any $x\in \mathbb Z^d$, with $d\ge 1$, and any set $A\subseteq \mathbb Z^d$, one has 
$$\mathbb P(\mathcal T_-^x \cap A \neq \emptyset)  =
\sum_{y\in A} \mathbb E\Big[ \mathbf 1\{\mathcal T_-^y \cap A = \emptyset\}\cdot \mathcal L_+(y,x)\Big], $$ 
where 
$$\mathcal L_+(y,x) = \sum_{u\in \mathcal T_+\smallsetminus\{\emptyset\}} \mathbf 1\{S_u^y = x\}. $$ 
\end{lem}
\begin{proof}
The proof is an immediate application of the shift invariance of the tree $\mathcal T$, first identified by Le Gall and Lin~\cite{LGL15,LGL16}, see also~\cite{Zhu16,BW20}. 
More precisely denote by 
$\sigma$ the last time in the past when the walk visits $A$, and for $n\in \mathbb Z$, denote with a slight abuse of notation by $S^x_n$ the position of the walk $S^x$ at the vertex with label $n$. 
Then by shift invariance, 
\begin{align*}
\mathbb P(\mathcal T_-^x \cap A \neq \emptyset )  & = \sum_{n=-1}^{-\infty} \sum_{y\in A} \mathbb P(\sigma = n, S_n^x = y)  = \sum_{n=-1}^{-\infty} \sum_{y\in A} 
\mathbb P( S_n^x = y, \, S_m^x \in A^c,  \text{ for all } m < n) \\
& = \sum_{n=1}^{\infty} \sum_{y\in A} \mathbb P(S_n^y = x,\mathcal T^y_- \cap A = \emptyset) =  \sum_{y\in A} \mathbb E\Big[ \mathbf 1\{\mathcal T_-^y \cap A = \emptyset\}\cdot \mathcal L_+(y,x)\Big]. 
\end{align*}
\end{proof}

\subsection{Green's functions}
Recall that the random walk Green's function is defined by 
$$g(x,y) = \mathbb E_x\Big[\sum_{n\ge 0} \mathbf 1\{X_n=y\}\Big]=g(0,y-x), $$ 
with $(X_n)_{n\ge 0}$ a simple random walk. We also let $g(z) = g(0,z)$, and recall that for $d\ge 3$, as $\|z\|\to \infty$ (see~\cite{LL10}), 
\begin{equation}\label{eq.g}
g(z) \sim \frac{a_d}{\|z\|^{d-2}},
\end{equation} 
where 
$$a_d = \frac d2\Gamma(\frac d2 - 1) \pi^{-d/2}. $$ 
We now define 
$$G(z) = \sum_{x\in \mathbb Z^d} g(x-z) g(x). $$ 
We shall need a few facts about this function. First, for any $x,z\in \mathbb Z^d$ (see e.g.~\cite{ASS23}), 
\begin{equation}\label{hit.G}
\mathbb P(z\in \mathcal T^x_-) \lesssim G(z-x). 
\end{equation}
The next result gives the leading order term in the asymptotic behavior of $G$ at infinity. 
\begin{lem}\label{lem.G}
Assume $d\ge 5$. Then as $\|z\|\to \infty$, 
$$G(z) \sim \frac{c_d}{\|z\|^{d-4}}, $$
with $c_d= \frac {d^2}{2(d-4)} \cdot \pi^{-d/2} \cdot \Gamma(\frac d2-1)$. 
\end{lem}
\begin{proof} 
One has using~\eqref{eq.g} and rotational invariance, $G(z) \sim c_d \cdot \|z\|^{4-d}$, with 
$$c_d=a_d^2\cdot \int_{\mathbb R^d} \frac 1{\|y-u\|^{d-2}} \cdot \frac 1{\|y\|^{d-2}} \, dy, $$
for any $u$ with $\|u\|= 1$. Note that by integrating over the unit sphere $\mathcal S(0,1)$, we find 
$$c_d= \frac{a_d^2}{|\mathcal S(0,1)|}\cdot \int_{\mathbb R^d} \frac 1{\|y\|^{d-2}} \left(\int_{\mathcal S(0,1)} \frac 1{\|y-u\|^{d-2}} \, du\right)\, dy. $$ 
Using next that $z\mapsto \|y\|^{2-d}$ is harmonic on $\mathbb R^d\smallsetminus \{0\}$, we find that 
\begin{equation*}
\frac 1{|\mathcal S(0,1)|}\int_{\mathcal S(0,1)} \frac 1{\|y-u\|^{d-2}} \, du =
\left\{
 \begin{array}{ll}
 \|y\|^{2-d} & \text{if } \|y\|>1\\
 1 & \text{if }\|y\|<1,
 \end{array}
 \right. 
 \end{equation*}
and a change of variables in polar coordinates then yields
$$c_d= a_d^2\cdot \frac{2\pi^{d/2}}{\Gamma(d/2)}\left(\int_0^1 r\, dr + \int_1^{\infty} r^{3-d}\, dr\right) =a_d^2 \cdot \frac{2\pi^{d/2}}{\Gamma(d/2)}\cdot \frac{d-2}{2(d-4)},$$
which after simplifying gives the desired result.  
\end{proof}
Finally one should need the following gradient bound. 
\begin{lem}\label{lem.G.grad}
Assume $d\ge 5$. 
One has for any $z,h\in \mathbb Z^d$, with $\|h\|\le \|z\|/2$, 
$$G(z+h) = G(z)\cdot \Big(1 + \mathcal O\big(\frac{\|h\|}{\|z\|}\big)\Big). $$  
\end{lem}
\begin{proof}
The result for the function $g$ is already known, see e.g.~\cite[Theorem 4.3.1]{LL10}, even for all $h$ satisfying $\|h\|\le \tfrac 23 \|z\|$. Injecting this in the definition of $G$, we get for $\|h\|\le \|z\|/2$, 
\begin{align*}
G(z+h) &  = \sum_{u\in \mathbb Z^d} g(z+h+u)g(u) = \sum_{v\in \mathbb Z^d} g(z-v) g(h+v)\\
& = \sum_{\|h\|\le (2/3)\|v\|} g(z-v)g(v)\big(1+\mathcal O(\frac{\|h\|}{\|v\|})\big) + \sum_{\|h\|>(2/3)\|v\|} g(z-v)g(h+v)\\
& = G(z) \cdot \Big(1 + \mathcal O\big(\frac{\|h\|}{\|v\|}\big)\Big) +  \sum_{\|h\|>(2/3)\|v\|} g(z-v)g(h+v) - \mathcal O(1)\cdot\sum_{ \|h\|>(2/3)\|v\|} g(z-v)g(v)\\
&  =  G(z) \cdot \Big(1 + \mathcal O\big(\frac{\|h\|}{\|v\|}\big)\Big) + \mathcal O(\|h\|^2\cdot g(z)) = G(z) \cdot \Big(1 + \mathcal O\big(\frac{\|h\|}{\|v\|}\big)\Big). 
\end{align*}
\end{proof}


\subsection{Variational characterization of the branching capacity}
We state here a result from~\cite{ASS23} that we shall use only in dimension $5$ for proving the lower bound in Proposition~\ref{m5}. It shows that 
the branching capacity is of the same order as the inverse of an energy.  
\begin{thm}[\cite{ASS23}] \label{thm.ASS2}
Assume $d\ge 5$. There exist positive constants $c$ and $C$, such that for any nonempty finite set $A\subset \mathbb Z^d$, 
$$\frac c{\bca(A)} \le \inf\Big\{\sum_{x,y\in A} G(x-y) \, \nu(x)\nu(y) : \nu \textrm{ probability measure on } A\Big\} \le \frac {C}{\bca(A)}. $$ 
\end{thm}
In particular the inverse of the middle term in the above display could provide an alternative definition of the branching capacity, which would be more intrinsic, in that 
it would not depend on a particular choice of critical probability measure $\mu$. However, it is not clear if with this definition, the law of large numbers would still hold in dimension $6$; at least the proof given here would break completely. 


\subsection{Quantitative bounds on hitting probability}
Our goal here is to prove some quantitative bounds, given a finite set $A\subset \mathbb Z^d$, on the speed of convergence toward $\bca(A)$ of the probability that an infinite tree-indexed random walk starting from $z$ hits $A$, as $\|z\|\to \infty$, when conveniently normalized. We only state the result in dimension $6$ for convenience, as we shall only need it in this case, but analogous bounds could be proved in any dimension $d\ge 5$, with the same arguments.

We define the diameter of a finite set $A$ as ${\rm diam}(A) = \max\{\|x-y\| : x,y\in A\}$.
\begin{lem}\label{lem.hit}
For any finite set $A\subset \mathbb Z^6$, containing the origin, and any $x$, satisfying $\|x\|\ge  {\rm diam}(A)^6$, 
$$\frac{\mathbb P(\mathcal T_-^x \cap A \neq \emptyset)}{G(x)} = \frac{\sigma^2}{2}\cdot \bca(A)+ \mathcal O\Big( \frac{{\rm diam}(A)^3}{\|x\|^{2/3}}\Big).$$ 
\end{lem}
\begin{proof}
By Lemma~\ref{lem.lastpassage} one has with the notation thereof, 
$$\mathbb P(\mathcal T_-^x \cap A \neq \emptyset) =\sum_{y\in A} \mathbb E\Big[\mathbf 1\{\mathcal T^y_-\cap A=\emptyset\} \cdot \mathcal L_+(y,x)\Big].$$
If the two terms in the expectation above were independent, we would be done, because one can observe that (see e.g.~\cite{ASS23}),  
\begin{equation}\label{expected.L+}
\mathbb E[\mathcal L_+(y,x)] = \frac{\sigma^2}{2} G(x-y) + \mathcal O(g(x-y)),
\end{equation}
and thus the result would follow directly from Lemma~\ref{lem.G.grad}. The problem is of course that they are not independent, thus our goal will be to decorrelate them as much as possible.

To this end, fix some $y\in A$, as well as a tree indexed walk starting from $y$, and define
$$\tau^y_r= \inf \big\{k\ge 0 : \widetilde X_k \in \partial B(0,r)\big\},$$
with $2\, \textrm{diam}(A)\le r\le \|x\|/2$, to be fixed later, and where we recall that $\widetilde X$ refers to the walk indexed by the spine of the tree $\mathcal T$. 
For $0\le a\le b\le \infty$, we let $\mathcal F_+^y[a,b]$ and $\mathcal F_-^y[a,b]$ denote the forests of trees respectively in the future and the past of $\mathcal T^y$ hanging off the spine at vertices with intrinsic label between $a$ and $b$. Then let 
$$\mathcal L_+^1(y,x)  = \sum_{u\in \mathcal F_+^y[0,\tau^y_r]} \mathbf 1\{S_u^y = x\},\quad {\rm and}\quad \mathcal L_+^2(y,x)  = \sum_{u\in \mathcal F_+^y[\tau^y_r+1,\infty)} \mathbf 1\{S_u^y = x\}. $$ 
One has for any $y\in A$, 
\begin{equation}\label{two.terms}
\mathbb E\Big[\mathbf 1\{\mathcal T^y_-\cap A=\emptyset\} \cdot \mathcal L_+(y,x)\Big] =\mathbb E\Big[\mathbf 1\{\mathcal T^y_-\cap A=\emptyset\} \cdot \mathcal L_+^1(y,x)\Big] +   \mathbb E\Big[\mathbf 1\{\mathcal T^y_-\cap A=\emptyset\} \cdot \mathcal L_+^2(y,x)\Big]. 
\end{equation}
We upper bound the first term using Lemma~\ref{lem.G.grad} as follows 
$$
\mathbb E\Big[\mathbf 1\{\mathcal T^y_-\cap A=\emptyset\} \cdot \mathcal L_+^1(y,x)\Big]  \le \mathbb E\Big[\mathcal L_+^1(y,x)\Big]\lesssim g(x)\cdot \mathbb E[\tau_r^y] 
\lesssim g(x) \cdot r^2,  
$$
using for the last inequality the well-known fact that for a simple random walk, the expected time needed to reach $\partial B(0,r)$ is of order at most $r^2$.  
The second term in the right hand side of~\eqref{two.terms}, which is the dominant part, will be evaluated using the independence of the forests before and after time $\tau^y_r$, conditionally on the position of $\widetilde X$ at this time. 
More precisely, we first note that
\begin{align}\label{hit.2terms}
 & \mathbb E\Big[\mathbf 1\{\mathcal T^y_-\cap A=\emptyset\} \cdot \mathcal L_+^2(y,x)\Big]  \\
\nonumber =&  \mathbb E\Big[\mathbf 1\{\mathcal F^y_-[0,\tau_r^y]\cap A=\emptyset\} \cdot \mathcal L_+^2(y,x)\Big] -\mathbb E\Big[\mathbf 1\{\mathcal F^y_-[0,\tau_r^y]\cap A=\emptyset, \mathcal F^y_-[\tau_r^y+1,\infty) \cap A\neq \emptyset\} \cdot \mathcal L_+^2(y,x)\Big]. 
\end{align}
Then concerning the first term on the right hand side, using \eqref{expected.L+} and Lemma~\ref{lem.G.grad}, we already get
$$\mathbb E\Big[\mathbf 1\{\mathcal F^y_-[0,\tau_r^y] \cap A=\emptyset\} \cdot \mathcal L_+^2(y,x)\Big] = \frac{\sigma^2}{2} G(x)\cdot \Big(1+\mathcal O(\frac{r}{\|x\|})\Big) 
\cdot \mathbb P( \mathcal F^y_-[0,\tau_r^y] \cap A=\emptyset) . $$ 
Moreover, by a result of Zhu~\cite{Zhu16}, one has (recall that the dimension is equal to $6$ here),  
$$\mathbb P(\mathcal F^y_-[\tau_r^y,\infty) \cap A\neq \emptyset) \lesssim \frac{\bca(A)}{r^2}\lesssim \frac{\textrm{diam}(A)^2}{r^2}.$$
In particular by choosing $r$ large enough, one can always ensure that the probability on the left hand side is smaller than $1/2$. This has for consequence that  
\begin{equation}\label{hit.eA}
\mathbb P( \mathcal F^y_-[0,\tau_r^y] \cap A=\emptyset)  = e_A(y) \cdot \Big(1 + \mathcal O\big(\frac{\textrm{diam}(A)^2}{r^2}\big)\Big). 
\end{equation}
Altogether this gives 
$$\mathbb E\Big[\mathbf 1\{\mathcal F^y_-[0,\tau_r^y] \cap A=\emptyset\} \cdot \mathcal L_+^2(y,x)\Big] = \frac{\sigma^2}{2} G(x)\cdot e_A(y)\cdot \Big\{1 + \mathcal O\Big(\frac{r}{\|x\|} + \frac{\textrm{diam}(A)^2}{r^2}\Big)\Big\}. $$
Now it remains to consider the second term in~\eqref{hit.2terms}. By~\eqref{hit.eA}, one has 
$$\mathbb E\Big[\mathbf 1\{\mathcal F^y_-[0,\tau_r^y]\cap A=\emptyset, \mathcal F^y_-[\tau_r^y,\infty) \cap A\neq \emptyset\} \cdot \mathcal L_+^2(y,x)\Big]
\lesssim e_A(y) \cdot \sup_{z\in \partial B(0,r)} \mathbb E\Big[\mathbf 1\{\mathcal T^z_- \cap A\neq \emptyset\} \cdot \mathcal L_+(z,x)\Big]. $$
We let $r_0 = 2\, \textrm{diam}(A)$, and first upper bound the time spent at $x$ after the spine hits $\partial B(0,r_0)$ if it ever happens. 
Using some independence and~\eqref{expected.L+}, we get that for any $z\in \partial B(0,r)$, 
$$\mathbb E\Big[\mathbf 1\{\tau_{r_0}^z<\infty\} \cdot \mathcal L_+(\widetilde X_{\tau_{r_0}^z},x)\Big] \lesssim \frac{g(z)}{g(r_0)} \cdot G(x),$$
with the notation $g(s) = s^{-4}$, for $s>0$.  
It amounts next to upper bound the time spent at $x$ in the future before the spine hits $\partial B(0,r_0)$ under the event that the past hits $A$. 
We shall first condition on the positions of the walk indexed by the spine, and use a union bound for the probability that the past hits $A$. We know by~\cite{Zhu16} that 
for any $v$, satisfying $\|v\|>r_0$, 
the probability for a critical tree indexed walk starting from $v$ to hit $A$ is of order at most $g(v)\cdot \bca(A)$. Therefore, denoting by $\ell^z(u)$ the time spent at $u$ by the walk $\widetilde X$ starting from $z$, we get that for any $z\in \partial B(0,r)$, 
\begin{align*}
& \mathbb E\Big[\mathbf 1\{\mathcal T^z_- \cap A\neq \emptyset\} \cdot \mathcal L_+(z,x)\Big] \lesssim 
\sum_{\|u\|,\|v\|>r_0}  g(x-u)\cdot \mathbb E[\ell^z(u) \ell^z(v)] \cdot g(v)\cdot  \bca(A) + \frac{g(z)}{g(r_0)} \cdot G(x) \\
& \lesssim \sum_{\|u\|,\|v\|>r_0}  g(x-u)\cdot g(u-v)\cdot(g(z-u)+g(z-v)) \cdot g(v)\cdot  \bca(A) + \frac{g(z)}{g(r_0)} \cdot G(x)\\
&\lesssim \Big(G(z)\cdot G(x) + g(x)\cdot \log \|x\|\Big) \cdot \bca(A) +  \frac{g(z)}{g(r_0)} \cdot G(x).
\end{align*}
Summing now all these estimates over $y\in A$, we conclude that  
$$\frac{\mathbb P(\mathcal T_-^x \cap A \neq \emptyset)}{G(x)}  = \frac{\sigma^2}{2}\cdot \bca(A) \cdot \Big(1+ \mathcal O(\frac{r}{\|x\|} + \frac{\textrm{diam}(A)^2}{r^2})\Big) + \mathcal O(|A|\cdot G(x)\cdot r^2),$$
at least if we assume that $\tfrac{\log \|x\|}{\|x\|^2}\le \tfrac{1}{r^2}$.  
Now take $r=(\|x\|\cdot \textrm{diam}(A)^2)^{1/3}$. 
Using that $\bca(A) \le \textrm{diam}(A)^2$, and the rough bound $|A| \le \textrm{diam}(A)^6$, we get well the desired result.   
\end{proof}

\subsection{Lawler's identity and first consequences}
We give here an analogous version for the branching capacity of a wonderful identity discovered by Lawler in the setting of Newtonian capacity, see e.g.~\cite{L91}, 
and which has also been used successfully by Bai and Wan when studying the capacity of a branching random walk in the recent work~\cite{BW20}. 

\vspace{0.2cm}
For $n\ge 1$, let $\xi_n^l$ and $\xi_n^r$ be two independent Geometric random variables with parameter $1/n$. Let $(X_k)_{k\in \mathbb Z}$ be a two-sided 
simple random walk starting from the origin at time $0$. Then for $a\le b \in \mathbb Z$, define $\mathcal R[a,b] = \{X_a,\dots,X_b\}$, and let 
$$e_n := \mathbf 1\big\{0\notin \mathcal R[1,\xi_n^r]\big\}.$$
Now consider an infinite invariant tree $\mathcal T$ independent of the walk $X$, and let $\mathcal A_n$ be the event 
\begin{equation*}
\mathcal A_n= \big\{\mathcal T^0_- \cap  \mathcal R[-\xi_n^l,\xi_n^r]= \emptyset \big\}.
\end{equation*}
Write also, with the notation of Lemma~\ref{lem.lastpassage},  
$$\mathcal L_n = \sum_{j=-\xi_n^l}^{\xi_n^r} \mathcal L_+(0,X_j). $$
\begin{lem}\label{lem.key.lawler}
Assume $d\ge 5$. Then for any $n\ge 1$, 
$$\mathbb E\Big[ \mathbf 1_{\mathcal A_n} \cdot e_n \cdot \mathcal L_n\Big] = 1.$$
\end{lem} 
\begin{proof}
For $m\ge 0$, and a nearest neighbor path $(x_1,\dots,x_m)$, define the event 
$$B(m,x_1,\dots,x_m) = \{\xi_n^l + \xi_n^r = m, X_{k-\xi_n^l} =  X_{-\xi_n^l} +  x_k, \text{ for }1\le k \le m\}. $$ 
Let also for $0\le j\le m$, 
$$B(m,j,x_1,\dots,x_m) = B(m,x_1,\dots,x_m) \cap \{\xi_n^l = j,\xi_n^r = m-j\}. $$ 
Note that all these events have the same probability, and thus for any $0\le j\le m$, 
$$\mathbb P\big(B(m,j,x_1,\dots,x_m)\big) = \frac{\mathbb P\big(B(m,x_1,\dots,x_m)\big) }{m+1}. $$ 
Thus one can write, with $x_0=0$,  
\begin{align*} 
\mathbb E\Big[ \mathbf 1_{\mathcal A_n} \cdot e_n \cdot \mathcal L_n\Big] 
 & = \sum_{m=0}^\infty \sum_{(x_1,\dots,x_m)}  \frac{\mathbb P(B(m,x_1,\dots,x_m))}{m+1}\sum_{j=0}^m \mathbb E\Big[ \mathbf 1_{\mathcal A_n} \cdot e_n \cdot \mathcal L_n \ \Big| \ B(m,j,x_1,\dots,x_m)\Big] \\
& =  \sum_{m=0}^\infty\sum_{(x_1,\dots,x_m)} \frac{\mathbb P(B(m,x_1,\dots,x_m))}{m+1}  \sum_{j=0}^m \mathbf 1\{x_k\neq x_j, \text{ for all }k>j\} \\
& \hspace{2cm} \times \mathbb E\Big[ \mathbf 1\{\mathcal T^{x_j}_- \cap \{0,x_1,\dots,x_m\}=\emptyset\} \cdot  \big(\sum_{\ell=0}^m \mathcal L_+(x_j,x_\ell) \big)\Big] \\  
& = \sum_{m=0}^\infty\sum_{(x_1,\dots,x_m)} \frac{\mathbb P(B(m,x_1,\dots,x_m))}{m+1} \sum_{\ell = 0}^m 
\sum_{j = 0}^m \mathbf 1\{x_k\neq x_j, \text{ for all }k>j\} \cdot \\
& \hspace{2cm}  \times \mathbb E\Big[ \mathbf 1\{\mathcal T^{x_j}_- \cap \{0,x_1,\dots,x_m\}=\emptyset\} \cdot  \mathcal L_+(x_j,x_\ell) \Big] \\   
& =  \sum_{m=0}^\infty\sum_{(x_1,\dots,x_m)} \mathbb P\big(B(m,x_1,\dots,x_m)\big)= 1, 
\end{align*}
using Lemma~\ref{lem.lastpassage} for the penultimate equality. 
\end{proof}

We now provide some first consequences of this lemma. 
Define for $n\in \mathbb N\cup\{\infty\}$, 
\begin{equation}\label{def.Un}
U_n =   \sum_{j=-\xi_n^l}^{\xi_n^r} \sum_{i\ge 0} d_i \cdot g(X_j,\widetilde X_i), \quad {\rm and} \quad 
Z_n = \sum_{j=-\xi_n^l}^{\xi_n^r} \sum_{i\ge 0} d_i \cdot \mathbf 1\{X_j=\widetilde X_i\}.
\end{equation}

\begin{cor}\label{cor.Lawler}
Assume $d\ge 5$. Then for all $n\ge 1$, 
$$\mathbb E\Big[ \mathbf 1_{\mathcal A_n} \cdot e_n \cdot (U_n-Z_n)\Big] = 1. $$
Moreover, if $d\ge 7$, then 
$$ \mathbb E\Big[ \mathbf 1_{\mathcal A_\infty} \cdot e_\infty \cdot (U_\infty-Z_\infty)\Big] = 1. $$ 
\end{cor}
\begin{proof}
The idea for the first identity is to start from the equation given by Lemma~\ref{lem.key.lawler}, and then condition with respect to the sigma-field
$$\mathcal G_n = \sigma\Big((d_i)_{i\ge 0}, (\widetilde X_i)_{i\ge 0}, (X_j)_{-\xi_n^l \le j\le \xi_n^r}\Big),$$ 
where $d_i$ is the number of offspring in the future of the $i$-th vertex on the spine, and we recall $(\widetilde X_i)_{i\ge 0}$ is the walk indexed by the spine of $\mathcal T$. 
The main observation is that conditionally on $\mathcal G_n$, 
the random variables $\mathcal L_n$ and $\mathbf 1_{\mathcal A_n}\cdot e_n$ are independent. Moreover, for each $i$, if we denote by $\mathcal T_i$ the adjoint critical tree hanging off the spine in the future at its $i$-th vertex, to which we remove the root, then for any $y\in \mathbb Z^d$, 
$$\mathbb E\Big[ \sum_{u\in \mathcal T_i } \mathbf 1\{S_u = y\}\mid \mathcal G_n\Big] = d_i\cdot \big(g(\widetilde X_i, y) - \mathbf 1\{\widetilde X_i = y\}\big),$$
since for a random walk indexed by a critical tree, starting from $x$, and conditioned by the fact that the root of the tree has exactly one offspring, the mean number of visits to a site $y$ is equal to $g(x,y)-\mathbf 1\{x=y\}$, if we do not count the starting point. Therefore, summing over $i\ge 0$, we get 
\begin{equation}\label{cond.Ln.Gn}
\mathbb E\big[ \mathcal L_n \mid \mathcal G_n\big] =  \sum_{i=0}^\infty \sum_{-\xi_n^l \le j\le \xi_n^r} d_i\cdot \big(g(\widetilde X_i, X_j) - \mathbf 1\{\widetilde X_i = X_j\}\big),
\end{equation}
which proves already the first claim of the corollary.

For the second claim note that by definition almost surely the sequence $\mathbf 1_{\mathcal A_n} \cdot e_n$ converges   
toward $\mathbf 1_{\mathcal A_\infty} \cdot e_\infty$, while $(U_n-Z_n)_{n\ge 0}$ converges toward $U_\infty-Z_\infty$. 
Moreover, for each $n\ge 1$, one has $0\le \mathbf 1_{\mathcal A_\infty} \cdot e_\infty\cdot (U_n-Z_n)\le U_\infty$, and if $d\ge 7$, by~\eqref{eq.g} and Lemma~\ref{lem.G}, 
$$\mathbb E[U_\infty]  \lesssim  \sum_{u,v\in \mathbb Z^d} g(u-v)g(u)g(v) =\sum_{u\in \mathbb Z^d} G(u)g(u)<\infty.$$ 
Thus the second claim follows from the first one and the dominated convergence theorem. 
\end{proof}


\section{Proof of Theorem \ref{m6}}\label{sec.6}
We assume in the whole section that $d=6$, and that $\mu$ has a finite third moment, which is in fact only needed for bounding the variance of $U_n$, defined in~\eqref{def.Un}.

\subsection{Concentration of the variable $U_n$}
We just state here our main estimates concerning the mean and variance of the variable $U_n$. 
The proof is postponed to Section~\ref{sec.proof.lem}, as it is a bit long and technical. 
\begin{prop}\label{prop.concentration}
One has as $n\to \infty$, 
$$\mathbb E[U_n]  \sim \frac{27\sigma^2}{2\pi^3} \cdot \log n, \quad \text{and}\quad 
{\rm Var}(U_n) = \mathcal O(\log n). $$   
\end{prop}

\subsection{Rough bounds on the probability of the event $\mathcal A_n$} 
We prove here some rough upper bound on the probability of the event $\mathcal A_n$, as well as on the event 
 $$\mathcal B_n = \left\{\mathcal T^0_- \cap \mathcal R[0,\xi_n^r] = \emptyset\right\}.$$ 
\begin{lem}\label{cor.An}
One has 
$$\mathbb P(\mathcal A_n) \lesssim \frac{1}{\log n}, \quad {\rm and}\quad \mathbb P(\mathcal B_n) \lesssim \frac{1}{\sqrt{\log n}} . $$  
\end{lem}
\begin{proof}
Recall the definitions of $U_n$ and $Z_n$ given in~\eqref{def.Un}, and let 
$$\mathcal E_n = \Big\{ |(U_n-Z_n) - \mathbb E[(U_n-Z_n)]|\ge \frac 12 \mathbb E[U_n]\Big\}.$$
Note that in any dimension $d\ge 5$, 
\begin{equation}\label{Z2}
\mathbb E[(Z_\infty)^2] = \sum_{x,y\in \mathbb Z^d} \sum_{\substack{i_1,i_2\ge 0\\ j_1,j_2\in \mathbb Z}} \mathbb P(\widetilde X_{i_1} = x,\widetilde X_{i_2} = y) \cdot \mathbb P(X_{j_1} = x,\widetilde X_{j_2} = y)  \lesssim  \sum_{x,y \in \mathbb Z^d} g(x)^2g(y-x)^2 \lesssim 1, 
\end{equation}
and hence by Proposition~\ref{prop.concentration} and Chebyshev's inequality, 
$$\mathbb P(\mathcal E_n) \lesssim \frac 1{\log n}.$$ 
Then using in addition Corollary~\ref{cor.Lawler}, we get 
\begin{equation}\label{bound.Anen}
\mathbb E\big[ \mathbf 1_{\mathcal A_n} \cdot e_n \big] \lesssim \mathbb P(\mathcal E_n) + \mathbb E\big[ \mathbf 1_{\mathcal A_n} \cdot e_n\cdot \mathbf 1_{\mathcal E_n^c}  \big] \lesssim \frac 1{\log n} +  \frac{\mathbb E\big[ \mathbf 1_{\mathcal A_n} \cdot e_n\cdot (U_n-Z_n)  \big] }{\mathbb E[U_n]} \lesssim \frac 1{\log n}. 
\end{equation}
We want now to remove $e_n$ from the expectation in the left-hand side. Denote by $\sigma$ the last visiting time of the origin by the walk $(X_k)_{k\ge 0}$. 
Let 
$$\mathcal A_n^\sigma = \big\{\mathcal T^0_- \cap \big(\mathcal R[-\xi_{\sqrt n}^l,0]\cup \mathcal R[\sigma,\sigma +  \xi_{\sqrt n}^r]\big) = \emptyset\big\}.$$ 
Since the law of the walk $X$ after time $\sigma$, is the law of a walk conditioned on not returning to the origin after time $0$, one has 
\begin{align*}
\mathbb P(\mathcal A_n)&  \le \mathbb P(\mathcal A_n^\sigma) + \mathbb P(\sigma+ \xi_{\sqrt n}^r > \xi_n^r) \le \mathbb E\big[\mathbf 1_{\mathcal A_{\sqrt n}} \cdot e_{\sqrt n}\big] + \mathbb P(\sigma+ \xi_{\sqrt n}^r > \xi_n^r) \\
& \stackrel{\eqref{bound.Anen}}{\lesssim} \frac 1{\log n} + \mathbb P(\sigma+ \xi_{\sqrt n}^r > \xi_n^r) \lesssim \frac 1{\log n}, 
\end{align*}
where the last bound follows from basic estimates. Indeed on one hand, 
$$\mathbb P(\sigma\ge  \sqrt n) \le \sum_{k\ge \sqrt n} \mathbb P(X_k = 0)\lesssim \sum_{k\ge \sqrt n} k^{-3} \lesssim \frac 1n,$$
and on the other hand, by standard properties of geometric random variables, 
$$\mathbb P(\xi_{\sqrt n}^r \ge \xi_n^r - \sqrt n) \le \mathbb P(\xi_{\sqrt n}^r \ge n^{3/4}) + \mathbb P(\xi_n^r \le 2n^{3/4}) \lesssim n^{-1/4}.$$ 
Thus so far we have proved the first inequality of the corollary. The second one follows by Cauchy-Schwarz inequality. Indeed, using also the independence between $\mathcal R[0,\xi_n^r]$ and $\mathcal R[-\xi_n^l,0]$, one deduces that  
$$\mathbb P(\mathcal B_n)^2 \le \mathbb E[\mathbb P(\mathcal B_n\mid \mathcal T^0_-)^2] = \mathbb P(\mathcal A_n).$$
\end{proof}


\subsection{Probability estimates of some non-intersection events}
Our main goal in this section is to prove estimates on some non-intersection events, which are simple consequences of Corollary~\ref{cor.An}. 
Denote by $\xi_n$ a Geometric random variable with parameter $1/n$, independent of everything else, and for $\varepsilon>0$, we denote by 
$x_\varepsilon$ the hitting point of $\partial B(0,1/\varepsilon)$ by the walk $\widetilde X$ indexed by the spine of $\mathcal T$. We start with the following estimate.
\begin{lem}\label{lem.nonintersection.eps}
For every $\varepsilon\in (0,1)$, there exists a constant $C(\varepsilon)>0$, such that for all $n\ge 2$,
$$ \mathbb P\Big(\mathcal T^{x_\varepsilon}_- \cap \mathcal R[0,\xi_n] = \emptyset\Big) \le \frac {C(\varepsilon)}{\sqrt{\log n}}. $$ 
\end{lem}

\begin{proof}
Define for $\varepsilon \in (0,1)$, 
$$ \widetilde \tau_\varepsilon =  \inf\{k\ge 0 : \widetilde X_k \in \partial B(0,1/\varepsilon) \}, $$ 
where we recall that $\widetilde X$ is the random walk indexed by the spine of $\mathcal T^0$. In particular,  $x_\varepsilon  = \widetilde X_{\widetilde \tau_\varepsilon}$, by definition. Now we let 
$\mathcal D_\varepsilon$ be the event that the path $\widetilde X$ up to time $\widetilde \tau_\varepsilon$ avoids $\mathcal R[0,\xi_n]$ and that none of the vertices on the spine up to time $\widetilde \tau_\varepsilon$ has any normal offspring. Note that there exists a constant $c(\varepsilon)>0$, such that for any $x\in \partial B(0,1/\varepsilon)$ and any path $\gamma$ starting from the origin, for which 
$$\mathbb P(\widetilde X_{\widetilde \tau_\varepsilon} = x, \mathcal R[0,\xi_n]= \gamma, \mathcal T^0_-\cap \mathcal R[0,\xi_n]=\emptyset)>0,$$
one also has 
$$\mathbb P(\mathcal D_\varepsilon, \widetilde X_{\widetilde \tau_\varepsilon} = x, \mathcal R[0,\xi_n]=\gamma) \ge c(\varepsilon)\cdot \mathbb P( \widetilde X_{\widetilde \tau_\varepsilon} = x, \mathcal R[0,\xi_n]=\gamma),$$
since in particular in this case  $x$ cannot be disconnected from the origin within $B(0,1/\varepsilon)$ by the path $\gamma$.   
Then  one has 
\begin{align*}
\mathbb P( \mathcal T^0_-\cap \mathcal R[0,\xi_n]=\emptyset) & \ge \mathbb P(\mathcal T^{x_\varepsilon}_- \cap \mathcal R[0,\xi_n] = \emptyset, \mathcal D_\varepsilon) \\
& = \sum_{x\in \partial B(0,1/\varepsilon)} \sum_{\gamma} \mathbb P(\mathcal T^x_- \cap \mathcal R[0,\xi_n] = \emptyset, \widetilde X_{\widetilde \tau_\varepsilon} = x, \mathcal R[0,\xi_n]=\gamma, \mathcal D_\varepsilon)\\
& \ge c(\varepsilon) \cdot  \sum_{x\in \partial B(0,1/\varepsilon)} \sum_{\gamma} \mathbb P(\mathcal T^x_- \cap \mathcal R[0,\xi_n] = \emptyset, \widetilde X_{\widetilde \tau_\varepsilon} = x, \mathcal R[0,\xi_n]=\gamma) \\
& = c(\varepsilon) \cdot \mathbb P(\mathcal T^{x_\varepsilon}_- \cap \mathcal R[0,\xi_n] =\emptyset), 
\end{align*}
and we conclude the proof using Corollary~\ref{cor.An}. 
\end{proof} 

We prove now a second estimate. Recall the definitions of $\widetilde \tau_\varepsilon$ given in the proof of the previous lemma. Then denote by $\mathcal F^0_-[0,\widetilde \tau_\varepsilon]$ the forest consisting of all the subtrees in the past of $\mathcal T^0$ hanging of the spine from vertices with intrinsic label between $0$ and $\widetilde \tau_\varepsilon$. 

\begin{lem}\label{prop.nonintersection.eps}
There exists a constant $C>0$, such that for every $\varepsilon\in (0,1)$, 
$$\limsup_{n\to \infty} \ \mathbb P(\mathcal F_-^0[0,\widetilde \tau_\varepsilon] \cap \mathcal R[0,\xi_n] = \emptyset ) \le \frac{C}{\sqrt{\log (1/\varepsilon)}}. $$ 
\end{lem}
\begin{proof}
Let $\xi_\varepsilon$ be a Geometric random variable with parameter $\varepsilon$, independent of everything else. Note that for any $n\ge \varepsilon^{-3}$, 
$$\mathbb P(\xi_\varepsilon> \xi_n) \le \mathbb P(\xi_\varepsilon\ge \varepsilon^{-2}) + \mathbb P(\xi_n \le \varepsilon^{-2}) \lesssim \varepsilon,$$
and thus one can always replace $\xi_n$ by $\xi_\varepsilon$ in the statement of the proposition. Moreover, for any $x\in \partial B(0,1/\varepsilon)$, 
$$\mathbb P(\mathcal T^x_- \cap \mathcal R[0,\xi_\varepsilon]\neq \emptyset ) \le \mathbb E\Big[|\mathcal T^x_- \cap \mathcal R[0,\xi_\varepsilon]|\Big]\lesssim \sum_{u\in \mathbb Z^6} 
G(x-u) \cdot \mathbb P(u\in \mathcal R[0,\xi_\varepsilon]) \lesssim \varepsilon. $$
Therefore, for $n$ large enough,
$$\mathbb P(\mathcal F_-^0[0,\widetilde \tau_\varepsilon] \cap \mathcal R[0,\xi_n] = \emptyset ) \lesssim \mathbb P(\mathcal T^0_- \cap \mathcal R[0,\xi_\varepsilon] = \emptyset) + \varepsilon, $$
and then the result follows from Corollary~\ref{cor.An}. 
\end{proof}

\subsection{Asymptotic of the mean}
Here we compute the leading order term in the asymptotic of the expectation of the branching capacity of the range.  
\begin{prop}\label{prop.mean}
One has 
$$\mathbb E\big[\mathbf 1_{\mathcal A_n} \cdot e_n\big] \sim \frac{2\pi^3}{27\sigma^2}\cdot \frac 1{\log n},$$
and 
$$\mathbb E\left[\bca(\mathcal R_n)\right] \sim \frac{2\pi^3}{27\sigma^2} \cdot \frac{n}{\log n}. $$  
\end{prop}
\begin{proof}
Let us start with the first claim of the proposition. 
By Lemma~\ref{cor.An} and~\eqref{Z2}, one has using also Cauchy-Schwarz inequality, 
$$
\mathbb E[\mathbf 1_{\mathcal A_n} \cdot e_n \cdot Z_n] \le \mathbb E[\mathbf 1_{\mathcal A_n} \cdot e_n]^{1/2} \cdot \mathbb E[Z_\infty^2]^{1/2} \lesssim \frac{1}{\sqrt{\log n}}. 
$$
Therefore Corollary~\ref{cor.Lawler} gives 
\begin{equation}\label{mean.d6.1} 
\mathbb E\Big[ \mathbf 1_{\mathcal A_n} \cdot e_n \cdot U_n\Big] = 1+\mathcal O\big(\frac 1{\sqrt{\log n}}\big).  
\end{equation}
Then letting $\overline U_n = U_n - \mathbb E[U_n]$, and using Proposition~\ref{prop.concentration}, we get 
$$\mathbb E\big[\mathbf 1_{\mathcal A_n} \cdot e_n\big] = \frac 1{\mathbb E[U_n]} - \frac{\mathbb E[\mathbf 1_{\mathcal A_n} \cdot e_n \cdot \overline U_n]}{\mathbb E[U_n]} 
+ \mathcal O\big(\frac 1{(\log n)^{3/2}}\big),$$ 
and it amounts now to bound the second term on the right hand side. 
For $\varepsilon>0$, let 
$$Y_\varepsilon = \mathbf 1\{|U_n-\mathbb E[U_n]|>\varepsilon \cdot \mathbb E[U_n]\}.$$ 
One has 
$$\mathbb E[\mathbf 1_{\mathcal A_n} \cdot e_n \cdot |\overline U_n|] \le \varepsilon\cdot \mathbb E[U_n] \cdot \mathbb E[\mathbf 1_{\mathcal A_n} \cdot e_n]
+ \mathbb E[\mathbf 1_{\mathcal A_n}  \cdot Y_\varepsilon \cdot |\overline U_n|], $$
and again it suffices to bound the second term on the right-hand side. By Cauchy-Schwarz inequality and Proposition~\ref{prop.concentration}, we have 
$$ \mathbb E[\mathbf 1_{\mathcal A_n}  \cdot Y_\varepsilon \cdot |\overline U_n|]\lesssim \mathbb E[\mathbf 1_{\mathcal A_n}  \cdot Y_\varepsilon]^{1/2} \cdot \sqrt{\log n}.$$
Now define 
$$U_n^+ =  \sum_{j=0}^{\xi_n^r} \varphi_{\widetilde X}(X_j), \quad {\rm and}\quad U_n^- =  \sum_{j=-\xi_n^\ell}^{-1} \varphi_{\widetilde X}(X_j).$$ 
Let also  
$$Y_\varepsilon^+ = \mathbf 1\{|U_n^+-\mathbb E[U_n^+]|>\varepsilon \cdot \mathbb E[U_n^+]\},\quad {\rm and}\quad Y_\varepsilon^- = \mathbf 1\{|U_n^--\mathbb E[U_n^-]|>\varepsilon \cdot \mathbb E[U_n^-]\}.$$
One has 
$$\mathbb  E[\mathbf 1_{\mathcal A_n}  \cdot Y_\varepsilon] \le \mathbb  E[\mathbf 1_{\mathcal A_n}  \cdot Y_\varepsilon^+]  + \mathbb  E[\mathbf 1_{\mathcal A_n}  \cdot Y_\varepsilon^-], $$ 
and by symmetry it suffices to bound the term $\mathbb  E[\mathbf 1_{\mathcal A_n}  \cdot Y_\varepsilon^+]$. We further decompose it into two terms as follows. 
Recall the definition of $\widetilde \tau_\varepsilon$ from the proof of Lemma~\ref{lem.nonintersection.eps},  and let 
$$U_n^{\varepsilon} = \sum_{j=0}^{\xi_n^r} \sum_{i=0}^{\widetilde \tau_{\varepsilon'}} d_i \cdot g(X_j,\widetilde X_i), \quad {\rm and}\quad \widetilde U_n^\varepsilon = \sum_{j=0}^{\xi_n^r} \sum_{i=\widetilde \tau_{\varepsilon}+1}^\infty d_i \cdot g(X_j,\widetilde X_i).$$
Now set $\varepsilon' = \exp(-\varepsilon^{-6})$, and define  
$$Y_\varepsilon^1 = \mathbf 1\Big\{ \big|U_n^{\varepsilon'}-\mathbb E[U_n^{\varepsilon'}]\big|>\frac {\varepsilon}{2} \cdot \mathbb E[U_n^+]\Big\}, \quad {\rm and} \quad 
Y_\varepsilon^2 = \mathbf 1\Big\{ \big|\widetilde U_n^{\varepsilon'}-\mathbb E[\widetilde U_n^{\varepsilon'}]\big|>\frac {\varepsilon}{2} \cdot \mathbb E[U_n^+]\Big\}. $$ 
One has 
$$\mathbb  E[\mathbf 1_{\mathcal A_n}  \cdot Y_\varepsilon^+]\le \mathbb  E[\mathbf 1_{\mathcal A_n}  \cdot Y_\varepsilon^1]+ \mathbb  E[\mathbf 1_{\mathcal A_n}  \cdot Y_\varepsilon^2]. $$ 
Moreover, letting $y_\varepsilon = \widetilde X_{\widetilde \tau_{\varepsilon'/4}}$, we can write using independence, 
\begin{align*}
\mathbb  E[\mathbf 1_{\mathcal A_n}  \cdot Y_\varepsilon^1]& \le \mathbb  E\Big[Y_\varepsilon^1\cdot \mathbf 1\big\{\mathcal T^{y_\varepsilon}_- \cap \mathcal R[-\xi_n^l,0]=\emptyset\big\} \Big]\\
& = \sum_{x\in \partial B(0,1/\varepsilon')} \sum_{y\in \partial B(0,4/\varepsilon')} \mathbb  E[Y_\varepsilon^1\cdot \mathbf 1\{\widetilde X_{\widetilde \tau_{\varepsilon'}}= x\}]
\cdot \mathbb P_x(\widetilde X_{\widetilde \tau_{\varepsilon'/4}}=y ) \cdot \mathbb P(\mathcal T^y_- \cap \mathcal R[-\xi_n^l,0]=\emptyset) \\
& \lesssim  \sum_{x\in \partial B(0,1/\varepsilon')} \sum_{y\in \partial B(0,4/\varepsilon')} \mathbb  E[Y_\varepsilon^1\cdot \mathbf 1\{\widetilde X_{\widetilde \tau_{\varepsilon'}}= x\}]
\cdot \mathbb P(\widetilde X_{\widetilde \tau_{\varepsilon'/4}}=y ) \cdot \mathbb P(\mathcal T^y_- \cap \mathcal R[-\xi_n^l,0]=\emptyset)\\
& = \mathbb E[Y_\varepsilon^1] \cdot \mathbb P(\mathcal T^{y_\varepsilon}_- \cap \mathcal R[-\xi_n^l,0]=\emptyset), 
\end{align*}
using also Harnack's inequality at the third line, see e.g.~\cite[Lemma 6.3.7]{LL10}. Now the same argument as the one used for proving Proposition~\ref{prop.concentration} shows that 
$\mathbb E[Y_\varepsilon^1] \lesssim \tfrac{1}{\varepsilon^2\log n}$. Using additionally Lemma~\ref{lem.nonintersection.eps}, then yields
$$ \mathbb  E[\mathbf 1_{\mathcal A_n}  \cdot Y_\varepsilon^1] \lesssim \frac {C(\varepsilon'/4)}{\varepsilon^2(\log n)^{3/2}}. $$  
Similarly one has, with the notation from Lemma~\ref{prop.nonintersection.eps}, and using this result,
\begin{align*}
\mathbb  E[\mathbf 1_{\mathcal A_n}  \cdot Y_\varepsilon^2]& \le \mathbb  E\Big[Y_\varepsilon^2\cdot \mathbf 1\{ \mathcal F^0_-[0,\widetilde \tau_{4\varepsilon'}] \cap \mathcal R[-\xi_n^\ell,0]=\emptyset\}\Big] \\
& \lesssim \mathbb  E[Y_\varepsilon^2]\cdot \mathbb P\big(\mathcal F^0_-[0,\widetilde \tau_{4\varepsilon'}] \cap \mathcal R[-\xi_n^\ell,0]=\emptyset\big) \\
 & \lesssim \frac 1{\sqrt{\log (1/\varepsilon')}} \cdot \frac 1{\varepsilon^2\log n} = \frac {\varepsilon}{\log n}. 
\end{align*}
Since this holds for all $\varepsilon \in(0,1)$, combining all the previous estimates proves that  
$$\mathbb  E\big[\mathbf 1_{\mathcal A_n} \cdot e_n\big] = \frac 1{\mathbb E[U_n]} + o\big(\frac 1{\log n}\big),$$ 
concluding the proof of the first part of the proposition, thanks again to Proposition~\ref{prop.concentration}.

For the second part we use first that 
$$\mathbb E[\bca(\mathcal R_n)] = \sum_{k=0}^n \mathbb P\big(\mathcal T^0_-\cap \mathcal R[-k,n-k]=\emptyset, 0\notin \mathcal R[1,n-k]\big). $$ 
Then the lower bound follows from the first part, by writing 
$$\mathbb E[\bca(\mathcal R_n)] \ge  n\cdot \mathbb P\big(\mathcal T^0_-\cap \mathcal R[-n,n]=\emptyset, 0\notin \mathcal R[1,n]\big)\ge n\cdot \mathbb E[\mathbf 1_{\mathcal A_{n(\log n)^2}}\cdot e_{n(\log n)^2}] - o(\frac n{\log n}),$$
using for the last inequality that for a Geometric random variable $\xi$ with parameter $p$, one has 
$\mathbb P(\xi \le n ) \le  p(n+1)$.  For the upper bound we write similarly, with $n' = n/(\log n)^2$, 
\begin{align*}
\mathbb E[\bca(\mathcal R_n)]  & \le 2n' + (n-2n') \cdot \mathbb P\big(\mathcal T^0_-\cap \mathcal R[-n',n']=\emptyset, 0\notin \mathcal R[1,n']\big)\\
& \le 2n' + (n-2n')\cdot \mathbb E[\mathbf 1_{\mathcal A_{n/(\log n)^4}}\cdot e_{n/(\log n)^4}] - o(\frac n{\log n}),  
\end{align*}
using this time that for a Geometric random variable $\xi$ with parameter $p$, one has $\mathbb P(\xi > n') \le (1-p)^{n'}$. This concludes the proof of the proposition. 
\end{proof}


\subsection{Conclusion}
We are now in position to conclude the proof of the weak law of large numbers in dimension $6$. 

\begin{proof}[Proof of Theorem~\ref{m6}]
To conclude the proof it suffices to show that 
\begin{equation}\label{Variance}
{\rm Var}(\bca(\mathcal R_n)) = o(\mathbb E[\bca(\mathcal R_n)]^2). 
\end{equation}
For this one can follow verbatim the proof of Chang~\cite{Chang17}, which we briefly recall for reader's convenience. 
Note first that by Lemma~\ref{lem.hit}, one has for any $z$ with $\|z\|\ge n^6$, 
$$\frac{\sigma^2}{2}\cdot \bca(\mathcal R_n) = \frac{\mathbb P(\mathcal T^z_- \cap \mathcal R_n\neq \emptyset \mid \mathcal R_n)}{G(z)} + \mathcal O\big(\frac 1n\big).$$ 
As a consequence, 
\begin{equation}\label{expected.mean}
\frac{\sigma^2}{2}\cdot \mathbb E[\bca(\mathcal R_n)] = \frac{\mathbb P(\mathcal T^z_- \cap \mathcal R_n\neq \emptyset)}{G(z)} + \mathcal O\big(\frac 1n\big),
\end{equation}
and with $\widetilde {\mathcal T}^z_-$ an independent copy of $\mathcal T^z_-$,
\begin{equation}\label{expected.square}
\frac{\sigma^4}{4}\cdot \mathbb E\big[\bca(\mathcal R_n)^2\big] = \frac{\mathbb P(\mathcal T^z_- \cap \mathcal R_n\neq \emptyset , \widetilde{\mathcal T}^z_-\cap \mathcal R_n\neq \emptyset)}{G(z)^2} + \mathcal O(1), 
\end{equation}
using that $\bca(\mathcal R_n) \le n+1$ to show that the error term is well $\mathcal O(1)$.  
Next, we define
$$\tau_1=\inf\{k: X_k\in \mathcal T^z_-\}, \quad {\rm and} \quad  \tau_2=\inf\{k: X_k\in \widetilde{\mathcal T}^z_-\}.$$
One has by symmetry, 
\begin{equation}\label{numerator.square}
\mathbb P(\mathcal T^z_- \cap \mathcal R_n\neq \emptyset , \widetilde{\mathcal T}^z_-\cap \mathcal R_n\neq \emptyset)\le 2\mathbb P(\tau_1\le \tau_2\le n). 
\end{equation}
Then by using the Markov property for the walk $X$, we get 
\begin{equation}\label{tau1tau2}
\mathbb P(\tau_1\le \tau_2\le n) = \mathbb E\left[\mathbf 1\{\tau_1\le n\} \cdot \mathbb P_{X_{\tau_1}}(\mathcal T^z_- \cap \mathcal R[0,n-\tau_1]\neq \emptyset\mid \tau_1)\right]. 
\end{equation} 
Letting $k(n)= (\sigma^2/2)\cdot\mathbb E[\bca(\mathcal R_n)]$, one has using also Lemma~\ref{lem.G.grad}, 
$$\mathbb P_{X_{\tau_1}}(\mathcal T^z_- \cap \mathcal R[0,n-\tau_1]\neq \emptyset\mid \tau_1) = G(z) \cdot k(n-\tau_1) + \mathcal O\big( \frac{G(z)}{n}\big). $$
Injecting this in~\eqref{tau1tau2}, and using~\eqref{expected.mean} gives 
\begin{equation}\label{tau1tau2bis}
\mathbb P(\tau_1\le \tau_2\le n) = G(z)\cdot \mathbb E\left[\mathbf 1\{\tau_1\le n\} \cdot k(n-\tau_1)\right] +\mathcal O\big(G(z)^2\big).
\end{equation}
The final step is to show that conditionally on the event $\{\tau_1\le n\}$, the random variable $\tau_1/n$ converges in law to a uniform random variable in $[0,1]$, as $n\to \infty$ (uniformly in $z$ with $\|z\|\ge n^3$). For this one can write using~\eqref{expected.mean} and Proposition~\ref{prop.mean}, that for any $s\in (0,1)$,  
$$\mathbb P(\tau_1\le ns\mid \tau_1\le n) = \frac{\mathbb P(\mathcal T^z_- \cap \mathcal R_{ns}\neq \emptyset)}{\mathbb P(\mathcal T^z_- \cap \mathcal R_n\neq \emptyset)} 
= \frac{k(ns)+\mathcal O(1)}{k(n) + \mathcal O(1)} = s\cdot (1 + o(1)). $$
Using this and~\eqref{tau1tau2bis}, as well as~\eqref{expected.mean} and Proposition~\ref{prop.mean}, yields
$$\mathbb P(\tau_1\le \tau_2\le n) = G(z)\cdot \mathbb E\left[\mathbf 1\{\tau_1\le n\} \cdot k(n-\tau_1)\right] = G(z)^2 \cdot k(n)^2\cdot (1+o(1)) + \mathcal O(G(z)^2),$$ 
and plugging this into~\eqref{numerator.square} and~\eqref{expected.square} concludes the proof of~\eqref{Variance}, and thus the proof of Theorem~\ref{m6} as well. 
\end{proof} 

\begin{rem}[Sketch of proof of the strong law of large numbers]\label{rem.stronglaw}
\emph{We now briefly explain how one can strengthen the weak law into a strong law of large numbers. The main point is to obtain a quantitative bound on the second order term in the asymptotic expansion of the expected branching capacity of the range. More precisely one needs a bound of the form 
\begin{equation}\label{secondorder}
\mathbb E[\bca(\mathcal R_n)] = \frac{2\pi^3}{27\sigma^2} \cdot \frac{n}{\log n} + \mathcal O\big(\frac{n}{(\log n)^{1+\delta}}\big), 
\end{equation}
for some $\delta>0$. Indeed, once this is obtained, then a careful look at the previous proof above reveals that this would yield a better bound on the variance, namely 
$$\frac{{\rm Var}(\bca(\mathcal R_n))}{\mathbb E[\bca (\mathcal R_n)]^2}  = \mathcal O\big(\frac 1{(\log n)^\delta}\big). $$
In turn, once such bound is known, then one can follow exactly the same proof as in~\cite{ASS19} to deduce almost sure convergence. Roughly, using a dyadic decomposition scheme, one can express the branching capacity of the range as a sum of independent and (almost) identically distributed terms, plus a sum of so-called crossed terms, whose variance is controlled. Hence, one has for any $L\ge 1$ a decomposition of the form  
$$\bca(\mathcal R_n) = \sum_{i=0}^{2^L-1} \bca (\mathcal R_n^{(i,L)}) + \sum_{\ell = 1}^L \sum_{j=0}^{2^{\ell-2}} \chi(\mathcal R_n^{(2j,\ell)},\mathcal R_n^{(2j+1,\ell)}), $$ 
where $\chi(A,B)$ is defined in the introduction, and $\mathcal R_n^{(j,\ell)} = \mathcal R[j \frac{n}{2^\ell},(j+1)\frac{n}{2^\ell}]$. Here, as we take $L$ of order $\log \log n$, the main contribution comes from the first sum, 
the second sum is shown to have a small variance, thanks to the previous bound. As a consequence one can deduce almost sure convergence of $\tfrac{\log n}{n}\cdot \bca(\mathcal R_n)$  along a subsequence of the form 
$a_n = \exp(n^{1-\delta'})$, for some $\delta'\in(0,1)$, just using Chebyshev's inequality and the Borel-Cantelli lemma. Finally using that $a_{n+1}/a_n$ converges to one, and monotonicity of the branching capacity, one easily extends this convergence along a subsequence into an almost sure convergence for the initial sequence. }

\emph{Thus the whole proof boils down to proving~\eqref{secondorder}, for some $\delta>0$. For this one can follow roughly the same strategy as in the proof of Proposition~\ref{prop.mean}, 
but with a different truncation of the variable $U_n^+$. In fact reproducing the same first steps, one can see that 
the main problem is to prove a bound of the form 
\begin{equation}\label{goal.SLLN}
\mathbb E \big[\mathbf 1_{\mathcal A_n} \cdot Y_n^+\big] = \mathcal O\big(\frac 1{(\log n)^\delta}\big), 
\end{equation}
where $Y_n^+ = \mathbf 1\big\{ |U_n^+-\mathbb E[U_n^+]|> (\log n)^{\frac 9{10}}\big\}$. 
To this end, fix some $L\ge 0$ and define 
$$V_n^L = \sum_{j=0}^{\xi_n^r} \sum_{i=0}^{\tau_L} d_i\cdot g(X_j,\widetilde X_i), \quad {\rm and}\quad W_n^L =  \sum_{j=0}^{\xi_n^r} \sum_{i=\tau_L}^\infty d_i\cdot g(X_j,\widetilde X_i),$$
where 
$$\tau_L= \inf \{k\ge 0 : \widetilde X_k\in \partial B(0,2^L)\}.$$ 
Write also $Y_n^1 = \mathbf 1\{|V_n^L - \mathbb E[V_n^L]|> \tfrac 12 (\log n)^{\frac 9{10}}\}$ and $Y_n^2 = \mathbf 1\{|W_n^L - \mathbb E[W_n^L]|> \tfrac 12 (\log n)^{\frac 9{10}}\}$, so that 
$$\mathbb E \big[\mathbf 1_{\mathcal A_n} \cdot Y_n^+\big]  \le \mathbb E \big[\mathbf 1_{\mathcal A_n} \cdot Y_n^1\big]  + \mathbb E \big[\mathbf 1_{\mathcal A_n} \cdot Y_n^2\big].$$ 
Now a similar proof as the one of Proposition~\ref{prop.concentration} can show that 
$${\rm Var}(V_n^L) \lesssim L, \quad {\rm and} \quad {\rm Var}(W_n^L) \lesssim \log n,$$ 
and thus $\mathbb E[Y_n^1] \lesssim L/(\log n)^{\frac{18}{10}}$, while $\mathbb E[Y_n^2] \lesssim (\log n)^{-4/5}$. 
Therefore a similar argument as in the original proof can show that if $L= \sqrt{\log n}$, then 
$$\mathbb E \big[\mathbf 1_{\mathcal A_n} \cdot Y_n^2\big] \le \mathbb E[\mathbf 1\big\{\mathcal F_-^0[0,\tau_{\sqrt{\log n}}] \cap \mathcal R[-\xi_n^l,0]=\emptyset\big\} \cdot Y_n^2\big] \lesssim (\log n)^{-\frac {21}{20}},$$
and on the other hand, one can simply write 
$$ \mathbb E \big[\mathbf 1_{\mathcal A_n} \cdot Y_n^1\big] \le \mathbb P(\mathcal A_n)^{1/2}\cdot  \mathbb E \big[Y_n^1\big]^{1/2} \lesssim (\log n)^{-\frac{23}{20}},$$
which give~\eqref{goal.SLLN} as wanted. Actually the last step is to show that one also has 
$$\mathbb E[U_n] =  \frac{27\sigma^2}{2\pi^3} \cdot \log n + \mathcal O(\sqrt{\log n}),$$
but this is more routine (yet slightly tedious) computation. For this one can follow the same argument as the one given in the next section, and use a finer asymptotic of the function $G$, itself following from finer asymptotic of the function $g$, which is well-known, see e.g.~\cite[Theorem 4.3.1]{LL10}. 
}
\end{rem}


\subsection{Proof of Proposition~\ref{prop.concentration}} \label{sec.proof.lem} 
We start by a preliminary result. Let 
$$G_n= \sum_{-\xi_n^l \le j\le \xi_n^r} G(X_j).$$ 
\begin{lem}\label{lem.Gn.asymp}
One has 
$$\mathbb E[G_n] \sim \frac{27}{\pi^3}\cdot \log n, \quad {\rm and}\quad {\rm Var}(G_n) \lesssim \log n. $$ 
\end{lem}
\begin{proof}
Let $\lambda = (1-1/n)$. One has by~\eqref{eq.g} and Lemma~\ref{lem.G},  
\begin{align*}
\mathbb E[G_n] & = G(0) + 2\sum_{j=1}^\infty \lambda^j \cdot \mathbb E[G(X_j)]  \sim  2\sum_{j=1}^{n} \mathbb E[G(X_j)]  \sim 2\sum_{\|u\|\le \sqrt n} G(u) g(u) \sim 2c_6 a_6 \pi^3 \int_1^{\sqrt n} \frac 1r \, dr \\
&= c_6 a_6 \pi^3 \log n = 27 \pi^{-3} \log n, 
\end{align*}
and we now deal with the variance. For this, unfortunately it does not seem possible to use an explicit computation as it was done in dimension four by Lawler, see the proof of~\cite[Proposition 3.4.1]{L91}, 
since the 
function $G$ is no longer harmonic when $d=6$. However, the heuristic argument given there still holds, and we shall use it here. More precisely, the idea is that parts of the trajectories of $X$ between times $2^i$ and $2^{i+1}$ are almost independent for different $i$'s. In order to formalize it, we introduce some more notation. First notice that by symmetry it suffices to bound the variance of $G_n^+ = \sum_{k=0}^{\xi_n} G(X_k)$, 
where $\xi_n$ is a Geometric random variable with parameter $1/n$, independent of the walk $X$. 
Then for $\ell \ge 0$, define $B_\ell = B(0,2^{\ell-1})$, and 
$$\tau_\ell = \inf\{k\ge 0 : X_k\in \partial B_\ell\}.$$ 
Let also 
$$Y_\ell^{(n)} = \sum_{\tau_\ell \wedge \xi_n \le k < \tau_{\ell + 1} \wedge \xi_n} G(X_k), \quad {\rm and}\quad Y_\ell = \sum_{\tau_\ell \le k < \tau_{\ell + 1} } G(X_k),$$ 
so that in particular, 
$$G_n^+ = \sum_{\ell =0}^\infty Y_\ell^{(n)}. $$ 
Recall that for a simple random walk, starting from $\partial B_\ell$, the probability to hit $B(0,r)$, for some $r<2^\ell$, is of order $g(2^\ell) / g(r)$, where we use the convention $g(r)=r^{2-d}$. It follows that 
 \begin{equation}\label{Yell.square.bound}
\sup_{\ell \ge 0} \sup_{x\in \partial B_\ell} \mathbb E_x[Y_\ell^2] \lesssim \sup_{\ell \ge 0} \sum_{\|u\|,\|v\| \le 2^\ell} G(u) G(u+v) g(2^\ell)^2  \lesssim 1. 
\end{equation}
Thus for any $\ell \ge 0$, by the Markov property, 
$$
\mathbb E[Y_\ell^{(n)}] \le \mathbb P(\xi_n \ge \tau_\ell)\cdot \sup_{x\in \partial B_\ell} \mathbb E_x[Y_\ell] \lesssim \mathbb P(\xi_n \ge \tau_\ell),
$$
 and likewise 
 \begin{equation}\label{Yell.square}
 \mathbb E[(Y_\ell^{(n)})^2] \le \mathbb P(\xi_n \ge \tau_\ell) \cdot \sup_{x\in \partial B_\ell} \mathbb E_x[Y_\ell^2]    \lesssim \mathbb P(\xi_n \ge \tau_\ell). 
 \end{equation}
Moreover, 
\begin{equation}\label{varGn}
{\rm Var}(G_n^+) =  \sum_{\ell \ge 0} {\rm Var}(Y_\ell^{(n)})  + 2\sum_{0\le \ell < m} \Big(\mathbb E[Y_\ell^{(n)}\cdot Y_m^{(n)}] -  \mathbb E[Y_\ell^{(n)}]\cdot\mathbb E[Y_m^{(n)}] \Big).
\end{equation}
The first sum above is handled using~\eqref{Yell.square}, which shows that it is bounded by 
$$\sum_{\ell \ge 0} \mathbb E[(Y_\ell^{(n)})^2] \lesssim \sum_{\ell \ge 0} \mathbb P(\xi_n\ge \tau_\ell) \lesssim \log n. $$ 
It amounts now to bound the second sum in~\eqref{varGn}. Define for $y\in B_m$ and $x\in \partial B_m$, 
$$H_m(y,x) = \mathbb P_y(X_{\tau_m} = x). $$ 
It is known, see e.g.~\cite[Proposition 6.4.4]{LL10}, that uniformly in $y \in B_\ell$, with $\ell \le m-1$,  
\begin{equation}\label{Hm.bound}
H_m(y,x) = H_m(0,x) (1+\mathcal O(2^{\ell - m})). 
\end{equation}
Now define 
$$L = \frac 12 \log_2(n)- 2\log_2(\log n), \quad {\rm and} \quad M = \frac 12 \log_2( n) - \log_2(\log n). $$ 
Write for $\ell,m\ge 0$, 
$$\Delta_{\ell,m}^{(n)} = \mathbb E[Y_\ell^{(n)}\cdot Y_m^{(n)}] -  \mathbb E[Y_\ell^{(n)}]\cdot\mathbb E[Y_m^{(n)}] . $$ 
We first bound for $\ell \le L$ and $m\ge M$, using the memoryless property of geometric random variables, 
\begin{align*}
\mathbb E[Y_\ell^{(n)}\cdot Y_m^{(n)}]  & \le \sum_{y\in \partial B_{\ell+1}} \sum_{x\in \partial B_M} 
\mathbb E\left[Y_\ell^{(n)} \mathbf 1\{X_{\tau_{\ell + 1}} = y\}\right]
\cdot H_M(y,x) \cdot \mathbb E_x[Y_m^{(n)}]  \\
& \le \mathbb E[Y_\ell^{(n)}]\cdot \sum_{x\in \partial B_M} H_M(0,x) (1+\mathcal O(2^{L-M}))\cdot \mathbb E_x[Y_m^{(n)}],
\end{align*} 
and for $m\ge M$, 
\begin{align*}
 \mathbb E[Y^{(n)}_m] & = \sum_{x\in \partial B_M} \mathbb P(\xi_n > \tau_M, X_{\tau_M} =x) \cdot \mathbb E_x[Y^{(n)}_m] \\
 & = \sum_{x\in \partial B_M}H_M(0,x) \cdot  \mathbb E_x[Y^{(n)}_m] - \sum_{x\in \partial B_M} \mathbb P(\xi_n \le \tau_M, X_{\tau_M} =x) \cdot \mathbb E_x[Y^{(n)}_m] \\
 & \ge  \sum_{x\in \partial B_M}H_M(0,x) \cdot  \mathbb E_x[Y^{(n)}_m] -  \mathbb P(\xi_n \le \tau_M) \cdot \sup_{x\in \partial B_M} \mathbb P_x[\xi_n \le \tau_m]. 
 \end{align*}
 Note also that 
 $$\mathbb P(\xi_n\le \tau_M) =\sum_{k\ge 0} \mathbb P(\xi_n = k) \cdot \mathbb P(\tau_M\ge k) \le \frac{\mathbb E[\tau_M]}{n} \lesssim \frac {2^{2M}}{n} \lesssim \frac 1{(\log n)^2}. $$
Altogether, this gives 
\begin{align*}
\sum_{\ell =0}^{L} \sum_{m\ge N}\Delta_{\ell,m}^{(n)} \lesssim \log n. 
\end{align*} 
Moreover, 
$$\sum_{L\le \ell < m} \mathbb E[Y_\ell^{(n)}\cdot Y_m^{(n)}]  \le \sum_{L\le \ell < m} \mathbb P(\xi_n \ge \tau_\ell) \cdot \sup_{x\in \partial B_{\ell+1}} \mathbb P_x(\xi_n\ge \tau_m) 
\lesssim (\log \log n)^2, $$
and it just remains to consider the case when $\ell < m \le M$. Note that the case $m=\ell +1$ can be handled using a similar bound as~\eqref{Yell.square}. 
Furthermore, if $\ell +2 \le m$, then by \eqref{Hm.bound},
\begin{align*}
\mathbb E[Y_\ell^{(n)}\cdot Y_m^{(n)}] & \le \mathbb E[Y_\ell\cdot Y_m] = \sum_{y\in \partial B_{\ell +1}} \sum_{x\in \partial B_m} \mathbb E[Y_\ell \cdot \mathbf 1\{X_{\tau_{\ell +1}} = y\}] \cdot H_m(y,x)\cdot \mathbb E_x[Y_m]  \\
& \le (1+ \mathcal O(2^{\ell - m}))\cdot \mathbb E[Y_\ell] \cdot \mathbb E[Y_m]. 
\end{align*}
Conversely, one can use that by Cauchy-Schwarz inequality,  
$$\mathbb E[Y_\ell^{(n)}] \ge \mathbb E[Y_\ell] - \mathbb E[Y_\ell^2]^{1/2} \cdot \mathbb P(\xi_n \le \tau_{\ell +1})^{1/2}\stackrel{\eqref{Yell.square.bound}}{\ge} 
 \mathbb E[Y_\ell] - C\cdot \mathbb P(\xi_n \le \tau_{\ell +1})^{1/2},  
$$ 
for some constant $C>0$, and that for $\ell \le M$, 
$$\mathbb P(\xi_n \le \tau_{\ell +1}) \le \frac{\mathbb E[\tau_{\ell +1}]}{n} \lesssim \frac {2^{2M}}{n} \lesssim \frac 1{(\log n)^2}, $$
which altogether give as well 
$$\sum_{0 \le \ell < m\le N}  \Delta_{\ell,m}^{(n)} \lesssim \log n. $$ 
This concludes the proof of the upper bound for the variance. 
\end{proof}

We now move to proving concentration for $U_n$. The proof is based on a similar idea.  
\begin{proof}[Proof of Proposition~\ref{prop.concentration}] 
First recall, see e.g.~\cite{ASS23}, that for each $i\ge 1$, and $k\ge 0$, one has $\mathbb P(d_i=k)=\sum_{j\ge k+1} \mu(j)$, and hence 
$$\mathbb E[d_i] =  \sum_{k\ge 1} (k-1) \sum_{j\ge k}\mu(j) = \sum_{j\ge 1} \frac{j(j-1)}{2}\mu(j) = \frac{\sigma^2}{2},$$
while $\mathbb E[d_0] =  1$. Thus, recalling~\eqref{def.Un} and the definition of $G$, we get
\begin{equation}\label{cond.exp.Un}
\mathbb E[U_n\mid X,\xi_n^r,\xi_n^l] = \frac{\sigma^2}{2}\cdot  G_n + (1-\frac{\sigma^2}{2}) g_n, 
\end{equation}
with 
$$g_n = \sum_{j=-\xi_n^l}^{\xi_n^r} g(X_j). $$ 
Therefore the result for the expectation of $U_n$ follows from Lemma~\ref{lem.Gn.asymp} together with the fact that $\mathbb E[g_n] \le 2G(0)$. 
We shall now use that 
\begin{equation}\label{var.cond.formula}
{\rm Var}(U_n) = \mathbb E\Big[{\rm Var}(U_n\mid X,\xi_n^l,\xi_n^r)\Big] + {\rm Var}\Big(\mathbb E[U_n \mid X,\xi_n^l,\xi_n^r]\Big). 
\end{equation}
Lemma~\ref{lem.Gn.asymp} and~\eqref{cond.exp.Un} yield
$$ {\rm Var}\Big(\mathbb E[U_n \mid X,\xi_n^l,\xi_n^r]\Big)\lesssim {\rm Var}(G_n) +  \mathbb E[g_n^2] \lesssim \log n + \mathbb E[g_n^2].$$
Furthermore, 
\begin{align*}
\mathbb E[g_n^2] & \le 4\mathbb E\Big[\Big(\sum_{j\ge 0} g(X_j)\Big)^2\Big] \le 8\sum_{0\le j\le k}\mathbb E[g(X_j)g(X_k)]  =8 \sum_{u,v\in \mathbb Z^d} g(u)^2g(u+v)g(v)   \lesssim 1,
 \end{align*}
and thus it only remains to consider the first term on the right-hand side of~\eqref{var.cond.formula}. This is where we use the hypothesis that $\mu$ has a finite third moment, 
which implies that $d_i$ has a finite second moment for all $i\ge 0$.

Now for $\ell \ge 0$, define 
$$\widetilde \tau_\ell = \inf\{k\ge 0 : \widetilde X_k \in \partial B_\ell\}, $$ 
and for $x\in \mathbb Z^6$, 
$$\widetilde Y_\ell(x) = \sum_{j=\widetilde \tau_\ell}^{\widetilde \tau_{\ell +1}-1 } d_j \cdot g(x,\widetilde X_j). $$ 
Note that by using~\eqref{Hm.bound}, one has for any $m\ge \ell +2$, uniformly over $x,y\in \mathbb Z^6$,  
$$\mathbb E[\widetilde Y_\ell(x)\cdot \widetilde Y_m(y)]\le (1+\mathcal O(2^{\ell - m})) \cdot \mathbb E[\widetilde Y_\ell(x)]\cdot \mathbb E[\widetilde Y_m(y)]. $$ 
Moreover, repeating the argument used for~\eqref{Yell.square.bound} yields for $m\ge 0$,  
$$\mathbb E[\widetilde Y_m(y)] \lesssim \sum_{u\in B_{m+1}} g(y,u) g(2^m) \lesssim \rho_m(y), $$ 
with  
$$\rho_m(y) = g(y) 2^{2m} \cdot \mathbf 1\{y\notin B_m\} + G(2^m)\cdot \mathbf 1\{y\in B_m\}.$$ 
As a consequence, letting 
$$h_\ell(y) = \frac{2^\ell}{(1+\|y\|)^3}\cdot \mathbf 1\{y\notin B_\ell\} + 2^{-2\ell} \cdot\mathbf 1\{y\in B_\ell\},$$
we get that for any $\ell \ge 0$, 
$$\sum_{m\ge \ell +1} 2^{\ell-m}\cdot  \mathbb E[\widetilde Y_m(y)] \lesssim h_\ell(y). $$ 
It follows that uniformly over $x,y\in \mathbb Z^6$, 
$$\sum_{\ell \ge 0}\sum_{m\ge \ell +2} \Big(\mathbb E[\widetilde Y_\ell(x)\cdot \widetilde Y_m(y)]-\mathbb E[\widetilde Y_\ell(x)]\cdot \mathbb E[\widetilde Y_m(y)] \Big) 
\lesssim \sum_{\ell \ge 0} \rho_\ell(x)\cdot h_\ell (y)\le \sum_{\ell \ge 0} h_\ell(x)\cdot h_\ell (y). $$ 
On the other hand, a similar computation as above yields 
$$\mathbb E\big[\widetilde Y_\ell(x)\cdot \big(\widetilde Y_\ell(y)+\widetilde Y_{\ell+1}(y)\big) \big]\lesssim \sum_{u,v\in B_{\ell+1}} g(x,u) g(y,u+v) g(2^\ell)^2 \lesssim \rho_\ell(x) \cdot \rho_\ell(y)
\le h_\ell(x) \cdot h_\ell(y).   
$$
Altogether this gives 
$$\sum_{\ell \ge 0}\sum_{m\ge 0} \Big(\mathbb E[\widetilde Y_\ell(x)\cdot \widetilde Y_m(y)]-\mathbb E[\widetilde Y_\ell(x)]\cdot \mathbb E[\widetilde Y_m(y)] \Big) 
\lesssim \sum_{\ell \ge 0} h_\ell(x) \cdot h_\ell(y). $$ 
From this we infer that 
$${\rm Var}(U_n\mid X,\xi_n^\ell, \xi_n^r) \lesssim \sum_{-\xi_n^\ell\le j,k\le \xi_n^r} \sum_{\ell \ge 0} h_\ell(X_j)\cdot h_\ell (X_k). $$  
Now, for any fixed $\ell \ge 0$, one has 
$$\mathbb E\left[\sum_{0\le j\le k\le \xi_n^r} h_\ell(X_j) h_\ell(X_k) \right]\lesssim \sum_{\|u\|,\|v\|\le \sqrt n} h_\ell(u) \cdot h_\ell(u+v)\cdot g(u)g(v)\lesssim \frac {n^2}{n^2+2^{4\ell}}, $$
and likewise,  
$$ \mathbb E\Big[\sum_{0\le j\le \xi_n^r} h_\ell(X_j) \Big]^2 \lesssim \left(\sum_{\|u\|\le \sqrt n} h_\ell (u) g(u)\right)^2\lesssim \frac {n^2}{n^2+2^{4\ell}}. $$ 
Thus as wanted, 
$$\mathbb E\Big[ {\rm Var}(U_n\mid X,\xi_n^\ell, \xi_n^r)\Big]\lesssim \sum_{\ell \ge 0} \frac {n^2}{n^2+2^{4\ell}} \lesssim \log n, $$ 
concluding the proof of the proposition. 
\end{proof}


\section{Proofs of Theorem~\ref{m7} and~Proposition \ref{m5}}\label{sec.57}

\begin{proof}[Proof of Proposition~\ref{m5}]
The proof is the same as in~\cite{ASS18}, which we recall for completeness. For the lower bound, we let 
$$L_n(x) = \sum_{k=0}^n \mathbf 1\{S_k = x\}, $$ 
and $\nu_n(x) = \tfrac{L_n(x)}{n+1}$, which defines a probability measure supported on $\mathcal R_n$. Thus one can use Theorem~\ref{thm.ASS2}, which gives that 
$$\bca(\mathcal R_n) \gtrsim \frac {n^2}{\sum_{x,y\in \mathcal R_n} G(x-y) L_n(x) L_n(y)}. $$
Then by using Cauchy-Schwarz's inequality, we get 
\begin{align}\label{lower.bound}
\nonumber \mathbb E[\bca(\mathcal R_n)]  & \gtrsim \frac {n^2}{\sum_{x,y\in \mathcal R_n} G(x-y) \mathbb E[L_n(x) L_n(y)]}\ge \frac {n^2}{\mathbb E[\sum_{0\le k\le n} \sum_{0\le \ell \le n} G(X_k-X_\ell)] }\\
&\gtrsim \frac {n}{\mathbb E[\sum_{0\le k\le n} G(X_k)] }\gtrsim \frac {n}{\sum_{\|u\|\le \sqrt n} G(u)g(u) }\gtrsim \sqrt n. 
\end{align}
The upper bound comes from the fact that the branching capacity is monotone for inclusion, and thus if $R_n = \max_{0\le k\le n} \| X_k\|$, then 
$\bca(\mathcal R_n) \le \bca(B(0,R_n)) \lesssim R_n$, as we know from~\cite{Zhu16} that the branching capacity of a ball of radius $R$ is of order $R$ in dimension $5$. 
Therefore $\mathbb E[\bca(\mathcal R_n)] \lesssim \mathbb E[R_n]$, and the desired upper bound follows since it is well known that $\mathbb E[R_n] \lesssim \sqrt n$. 
\end{proof}

\begin{proof}[Proof of Theorem~\ref{m7}]
The fact that the limit exists in~\eqref{LLN} follows from the ergodic theorem, exactly as in~\cite{JO69}. Let us recall the argument for reader's convenience.  
First one has 
$$\bca(\mathcal R_n) = \sum_{k=0}^n e_{\mathcal R_n}(X_k) \cdot \mathbf 1\big\{X_k\notin \{X_{k+1},\dots, X_n\}\big\}.$$
Thus, letting $\mathcal R_\infty$ and $\widetilde{\mathcal R}_\infty$ be two independent infinite ranges starting from the origin, one has 
$$\frac{\bca(\mathcal R_n)}{n} \ge \frac 1n \sum_{k=0}^n e_{\mathcal R_\infty\cup \widetilde{\mathcal R}_\infty }(X_k) \cdot \mathbf 1\big\{X_k\notin \{X_{k+1},\dots, \}\big\}, $$
and the ergodic theorem implies that the right hand side converges almost surely as $n\to \infty$, toward (with the notation of Corollary~\ref{cor.Lawler}) 
$$c_d = \mathbb E\Big[e_{\mathcal R_\infty\cup \widetilde{\mathcal R}_\infty}(0) \cdot \mathbf 1\big\{0\notin \mathcal R[1,\infty)\} \Big]= \mathbb E\Big[\mathbf 1_{\mathcal A_\infty}\cdot e_\infty\Big],$$
which provides already the lower bound 
$$\liminf_{n\to \infty}  \frac{\bca(\mathcal R_n)}{n}  \ge c_d.$$
To get the upper bound, notice that for any $n\ge 1$, 
$$\mathbb E[\bca(\mathcal R_n)] \le 2 \sqrt n + (n-2\sqrt n)\cdot \mathbb E\Big[e_{\mathcal R[0,\sqrt n]\cup \widetilde{\mathcal R}[0,\sqrt n]}\cdot \mathbf 1\{0\notin \mathcal R[0,\sqrt n]\}\Big],$$ 
and since by monotone convergence the expectation on the right hand side converges to $c_d$ as $n\to \infty$, it follows that 
$$\limsup_{n\to \infty} \frac{\mathbb E[\bca(\mathcal R_n)]}{n} \le c_d. $$  
Now fix some integer $m\ge 1$, and observe that by subadditivity of the branching capacity, see~\cite{Zhu16}, one has 
$$\bca(\mathcal R_n) \le \sum_{i = 0}^{\lfloor n/m\rfloor-1} \bca(\mathcal R[im,(i+1)m]).$$
Since the right-hand side is a sum of independent and identically distributed terms, one get by Kolmogorov's strong law of large numbers,  
$$\limsup_{n\to \infty}  \frac{\bca(\mathcal R_n)}{n} \le \frac{\mathbb E[\bca(\mathcal R_m)]}{m}. $$
Since this holds for any $m$, we obtain the converse inequality, 
$$\limsup_{n\to \infty}  \frac{\bca(\mathcal R_n)}{n}  \le \limsup_{m\to \infty} \frac {\mathbb E[\bca(\mathcal R_m)]}{m} \le c_d. $$ 
Finally, to see that $c_d$ is positive when $d\ge 7$, 
one can use the second statement of Corollary~\ref{cor.Lawler}. It has already been seen in its proof that $\mathbb E[U_\infty]$ is finite, which implies that $U_\infty$ is finite almost surely. Together with the second claim of Corollary~\ref{cor.Lawler}, we deduce that $\mathbf 1_{\mathcal A_\infty}\cdot e_\infty$ is not almost surely equal to zero, and thus $c_d>0$. 
\end{proof}


\begin{thebibliography}{99}

\bibitem{ASS18}  A. Asselah, B. Schapira, P. Sousi. Capacity of the range of random walk on $\mathbb Z^d$. Trans. Amer. Math. Soc. 370 (2018), 7627--7645.

\bibitem{ASS19} A. Asselah, B. Schapira, P. Sousi. Capacity of the range of random walk on $\mathbb Z^4$. Ann. Probab. 47 (2019), 1447--1497.

\bibitem{ASS23} A. Asselah, B. Schapira. P. Sousi. Branching capacity and local times of transient branching random walks, (2023), preprint. 

\bibitem{BW20}  T. Bai, Y. Wan. Capacity of the range of tree-indexed random walk. Ann. Appl. Probab. 32 (2022), 1557--1589.

\bibitem{BH22} T. Bai, Y. Hu. Convergence in law for the capacity of the range of a critical branching random walk, arXiv:2203.03188. 

\bibitem{Chang17} Y. Chang. Two observations on the capacity of the range of simple random walk on $\mathbb Z^3$ and $\mathbb Z^4$. Electron. Commun. Probab. 22, (2017). 

\bibitem{DE51} A. Dvoretzky, P. Erd\"os. Some problems on random walk in space. Proceedings Second Berkeley Symposium on Math. Statistics and Probability, 353--367. University of California Press, Berkeley (1951). 

\bibitem{JO69} N. C. Jain, S. Orey. On the range of random walk. Israel J. Math. 6 1968, 373--380 (1969). 

\bibitem{JP71} N. C. Jain, W. E. Pruitt. The range of transient random walk. J. Analyse Math. 24, (1971), 369--393.

\bibitem{L91} G. F. Lawler. Intersections of random walks. Second edition, Birkhauser, 1996.

\bibitem{LL10} G. F. Lawler, V. Limic.  Random walk: a modern introduction. Cambridge University Press, Cambridge, 2010. 

\bibitem{LG86} J.-F. Le Gall. Propri\'et\'es d'intersection des marches al\'eatoires. I. Convergence vers le temps local d'intersection. (French) [Intersection properties of random walks. I. Convergence to local time of intersection] Comm. Math. Phys. 104 (1986), 471--507.

\bibitem{LGL15}  J.-F. Le Gall, S. Lin. 
The range of tree-indexed random walk in low dimensions. Ann. Probab. 43 (2015), 2701--2728. 

\bibitem{LGL16} J.-F. Le Gall, S. Lin. The range of tree-indexed random walk. 
J. Inst. Math. Jussieu 15 (2016), 271--317. 

\bibitem{S20} B. Schapira. Capacity of the range in dimension 5. Ann. Probab. 48 (2020), 2988--3040.

\bibitem{Zhu16}  Q. Zhu. On the critical branching random walk I: branching capacity and visiting probability,  arXiv:1611.10324,


\end{thebibliography}
\end{document}